\documentclass[12pt]{amsart}
\usepackage{geometry}
\usepackage{hyperref}
\usepackage{amsthm,amsfonts,amsmath,amscd,amssymb,epsfig,verbatim,
}

\usepackage{graphicx}
\usepackage{epstopdf}
\title{The topology of rationally and polynomially convex domains}
\author{Kai~Cieliebak and Yakov~Eliashberg} 
\thanks{Y.~Eliashberg is partially supported by the NSF grant DMS-1205349} 
\address{Kai Cieliebak \\ Institut f\"ur Mathematik \\
    Universit\"at Augsburg \\
 Germany}
 \address{Yakov Eliashberg \\ Department of Mathematics \\   Stanford University \\ USA}

\date{}  
\let\oldmarginpar\marginpar
\renewcommand\marginpar[1]{\-\oldmarginpar[\raggedleft\footnotesize #1]%
{\raggedright\footnotesize #1}}

\theoremstyle{plain}
\newtheorem{theorem}{Theorem}[section]
\newtheorem{thm}[theorem]{Theorem}

\newtheorem{cor}[theorem]{Corollary}

\newtheorem{prop}[theorem]{Proposition}
\newtheorem{lemma}[theorem]{Lemma}
\newtheorem{criterion}[theorem]{Criterion}

\theoremstyle{remark}

\newtheorem{remark}[theorem]{Remark}
\newtheorem*{remark*}{Remark}

\newtheorem*{example*}{Example}

\theoremstyle{definition}

\newcommand{\id}{{\rm id}}

\newcommand{\wt}{\widetilde}
\newcommand{\wh}{\widehat}

\newcommand{\p}{\partial}

\newcommand{\om}{\omega}
\newcommand{\Om}{\Omega}
\newcommand{\eps}{\varepsilon}
\newcommand{\into}{\hookrightarrow}
\newcommand{\la}{\langle}
\newcommand{\ra}{\rangle}

\newcommand{\dbar}{\overline{\partial}}

\newcommand{\N}{{\mathbb{N}}}

\newcommand{\Z}{{\mathbb{Z}}}
\newcommand{\R}{{\mathbb{R}}}
\newcommand{\C}{{\mathbb{C}}}

\newcommand{\im}{{\rm im }}        
\newcommand{\st}{{\rm st}}

\newcommand{\Int}{{\rm Int\,}} 

\renewcommand{\min}{{\rm min}}
\renewcommand{\max}{{\rm max}}

\renewcommand{\Re}{{\rm Re\,}}
\renewcommand{\Im}{{\rm Im\,}}

\newcommand{\Id}{\mathrm {Id}}

\newcommand{\Grauert}{\mathrm{Grauert}}

\newcommand{\MM}{\mathcal{M}}

\newcommand{\OO}{\mathcal{O}}
\newcommand{\PP}{\mathcal{P}}
\newcommand{\RR}{\mathcal{R}}
\newcommand{\HH}{\mathcal{H}}

\def\Op{{\mathcal O}{\it p}\,}
\newcommand{\fW}{{\mathfrak W}}

\newcommand{\pbar}{{\bar\partial}}
\numberwithin{figure}{section}
\begin{document}


\begin{abstract}
We give in this article necessary and sufficient conditions on the
topology of a compact domain with smooth boundary in $\C^n$, $n\geq
3$, to be isotopic to a rationally or polynomially convex domain.
\end{abstract}

\maketitle

\section{Introduction}\label{sec:intro}

\subsection{Polynomial, rational and holomorphic convexity}\label{sec:def}
Recall the following  complex analytic notions of convexity for
domains in $\C^n$. For a compact set $K\subset \C^n$, one defines its
{\it polynomial hull} as 
$$
   \wh K_\PP := \{z\in\C^n\:\Bigl|\;
   |P(z)|\leq\mathop{\max}\limits_{u\in K}|P(u)|\hbox{ for  
  all complex polynomials }P \hbox{ on }\C^n\},
$$
and its {\it rational hull} as
$$
   \wh K_\RR := \{z\in\C^n\;\Bigl|\; |R(z)|\leq\mathop{\max}\limits_{u\in
     K}|R(u)|\hbox{ for all rational functions }R=\frac PQ,\; Q|_K\neq
   0\}. 
$$
Given an open set $U\supset K$, the {\it holomorphic hull of $K$ in
  $U$} is defined as 
$$
   \wh K^U_\HH := \{z\in U\;\Bigl|\;
   |f(z)|\leq\mathop{\max}\limits_{u\in K}|f(u)|\hbox{ for  
  all holomorphic functions }f \hbox{ on } U\}.
$$
A compact set $K\subset \C^n$ is called {\it rationally} (resp.~{\it
  polynomially}) {\em convex} if $\wh K_\RR=K$ (resp.~$\wh
K_\PP=K$). 
An open set $U\subset\C^n$ is called {\it holomorphically convex} if
$\wh K^U_\HH$ is compact for all compact sets $K\subset U$. A compact
set $K\subset \C^n$ is called {\it holomorphically convex} if it is
the intersection of its holomorphically convex open neighborhoods. We
have  
$$
   \text{polynomially convex} \Longrightarrow \text{rationally
     convex} \Longrightarrow \text{holomorphically convex}. 
$$
The first implication is obvious, while the second one follows from
the fact that by definition a rationally convex compact set $K$ is an
intersection of {\it bounded rational polyhedra} $\{|R_i|<c_1,
i=1,\dots, N\}$, where the $R_i$ are rational functions, and any
bounded rational polyhedron is clearly holomorphically convex.

Given a real valued function $\phi: U\to\R$ on an open subset
$U\subset\C^n$, we denote by $d^\C\phi:=d\phi\circ i$ its differential
twisted by multiplication with $i=\sqrt{-1}$ on $\C^n$, 
and we set $\om_\phi:=-dd^\C\phi=2i\p\dbar\phi$. A function $\phi$ is called {\em
$i$-convex} if $\om_\phi(v,iv)>0$
for all $v\neq 0$. A cooriented hypersurface $\Sigma\subset\C^n$ (of
real codimension $1$) is
called {\em $i$-convex} if there
exists an $i$-convex function $\phi$ defined on some neighborhood of
$\Sigma$ such that $\Sigma=\{\phi=c\}$, and $\Sigma$ is cooriented by
a vector field $v$ satisfying $d\phi(v)>0$. 

\begin{remark}  Traditionally $i$-convexity for functions is called
 {\it strict  plurisubharmonicity}, and $i$-convexity for cooriented
 hypersurfaces {\it strict pseudoconvexity}. We prefer to use the term
 $i$-convexity  for smooth functions and hypersurfaces (or more generally $J$-convexity in a manifold $V$
 with a complex  or even an almost complex  structure $J$)  
 not only because it is shorter, but also because it explicitly shows
 the dependence on the ambient complex structure.  Note that we are
 also ``downgrading" the traditional terminology from 
  the theory of functions of several complex variables: $i$-convexity
  corresponds to {\it strict} plurisubharmonicity or pseudoconvexity,
  while for non-strict plurisubharmonicity or pseudoconvexity we will
  use the term {\it weak $i$-convexity}.
\end{remark}

In this paper, by a {\em domain} we will always mean a {\it compact}
manifold  $W$ with
smooth boundary $\p W$, and by a {\em domain in $\C^n$} an embedded domain
$W\subset\C^n$ of real dimension $2n$.  In particular, a domain
$W\subset\C^n$ in our terminology is always a {\it closed subse}t.
According to a theorem of E.~Levi~\cite{Levi10}, any holomorphically 
convex domain $W\subset\C^n$ has {\it weakly} $i$-convex  boundary $\p
W$. The converse statement that 
the interior of any domain in $\C^n$ with weakly
$i$-convex boundary is holomorphically convex is known as the {\em Levi
  problem}. It was proved by K.~Oka~\cite{Oka53},
H.J.~Bremermann~\cite{Bre54}, and F.~Norguet~\cite{Nor54}. In a more
general context of  domains in Stein manifolds, the Levi problem was
resolved by H.~Grauert~\cite{Gra58} in the  i-convex case, and by  F.~Docquier and 
H.~Grauert  in the  weakly i-convex one \cite{DoGra60}.    
  
We call a domain $W\subset\C^n$ {\em $i$-convex} if its boundary is
$i$-convex. Note that any weakly $i$-convex domain in $\C^n$ can be
$C^\infty$-approximated by a slightly smaller $i$-convex one. 

\subsection{Topology of rationally and polynomially convex
  domains}\label{sec:intro-topology} 
We call a function $\phi:W\to\R$ on a domain $W$ {\it defining} if $\p
W$ is a regular level set of $\phi$ and $\phi|_{\p W}=\max_W\phi$. 
Any $i$-convex domain $W\subset \C^n$ admits a defining $i$-convex
function, so in particular it admits a defining Morse function without
critical points of index $>n$ (see e.g.~\cite{CieEli12}). It follows
that any holomorphically,
rationally or polynomially convex domain has the same property. It was
shown in~\cite{Eli90} (see Theorem 1.3.6 there and also~\cite[Theorem
  8.19]{CieEli12}) that for $n\geq 3$, any domain in $\C^n$ with such a
Morse function is smoothly isotopic to an $i$-convex one. 

The first result of this paper states that, for $n\geq 3$, there are no
additional constraints on the topology of rationally convex domains. 

\begin{thm}\label{thm:rational}
A compact domain $W\subset\C^n$, $n\geq 3$, is smoothly isotopic to a
rationally convex domain if and only if it admits a defining Morse
function without critical points of index $>n$. 
\end{thm}

Our second result gives necessary and sufficient constraints on the
topology of polynomially convex domains. 

\begin{thm}\label{thm:polynomial}
A compact domain $W\subset\C^n$, $n\geq 3$, is smoothly isotopic to a
polynomially convex domain if and only if it satisfies the following
topological condition: 
\begin{itemize}
\item[(T)] $W$ admits a defining Morse
function without critical points of index $>n$, and $H_n(W;G)=0$ for
every abelian group $G$.  
\end{itemize}
\end{thm}

The ``only if" part is well known and due to A.~Andreotti and
R.~Narasimhan \cite{AndNar62}, see also~\cite{For94} or
Remark~\ref{rem:T} below. Note that, in view of the universal coefficient
theorem, condition (T) is equivalent to the condition
\begin{itemize}
\item[(T')] $W$ admits a defining Morse function without critical
  points of index $>n$, $H_n(W )=0$, and $H_{n-1}(W)$ has no torsion.  
\end{itemize}

Theorems~\ref{thm:rational} and~\ref{thm:polynomial} are consequences of
the more precise Theorem~\ref{thm:flexible} below. 
Further analysis of condition (T) yields

\begin{prop}\label{prop:top}
(a) If $W$ is simply connected, then condition (T) is equivalent
to the existence  of a defining Morse function without critical points
of index $\geq n$. 

(b) For any $n\geq 3$ there exists a (non-simply connected) domain $W$
satisfying condition (T) with $\pi_n(W,\p W)\neq 0$.  
In particular, $W$ does not admit a defining function without critical
points of index $\geq n$.
\end{prop}

This paper was  motivated by the questions raised by
S.~Nemirovski whether every polynomially convex domain in $\C^n$ is 
{\em subcritical}, i.e., it admits a defining $i$-convex Morse function
without critical points of index $\geq n$, and whether there are any
additional constraints on the topology of rationally convex domains
besides the fact that they admit defining Morse functions without critical
points of index $>n$. Theorem \ref{thm:rational} provides a complete
answer to the latter question in the case $n\geq 3$, while 
Theorem~\ref{thm:polynomial} together with Proposition~\ref{prop:top}
(b) show that the answer to the former question is 
in general negative. In the simply connected case,
Proposition~\ref{prop:top} (a) provides a defining Morse function
without critical points of index $\geq n$, but we do not know whether
there exists such a function which is $i$-convex. See also the
discussion after Theorem~\ref{thm:flexible} below.

\subsection{Symplectic topology of rationally and polynomially convex
  domains}\label{sec:intro-sym-topology}
One may ask which {\em deformation classes} of Stein domains can be
realized as polynomially or rationally convex domains in $\C^n$. Here
by a {\em Stein domain} we mean a domain $W$ with an integrable
complex structure $J$ (making $W$ a complex manifold with boundary)
which admits a defining $J$-convex function $\phi:W\to\R$.
Two Stein domains $(W,J)$ an $(W',J')$ are called {\em deformation
  equivalent} if there exists a diffeomorphism $f:W\to W'$ such that
$f^*J'$ is Stein homotopic to $J$; see~\cite{CieEli12}.  We call a
diffeomorphism $f:W\to W'$ with this property a {\it deformation
  equivalence}. 

Recall from~\cite{EliGro91,CieEli12} that a {\em Weinstein domain
structure} on a domain $W$ is a triple $(\om,X,\phi)$ consisting of a
symplectic form $\om$ on $W$, a defining Morse function $\phi:W\to\R$,
and a vector field $X$ on $W$ which is Liouville for $\om$
(i.e.~$L_X\omega=\omega$) and gradient-like for $\phi$. 
 A {\em Weinstein homotopy} on a domain $W$ is a smooth 1-parameter
family $(\om_t,X_t,\phi_t)$, $t\in[0,1]$, of triples satisfying all
the conditions on Weinstein domain structures, except that we allow
the family $\phi_t$ to have birth-death critical points.   

Any defining $J$-convex Morse function $\phi$ on a Stein domain
$(W,J)$ induces a Weinstein structure
$\fW(W,J,\phi):=(\om_\phi,X_\phi,\phi)$ on $W$ where
$\om_\phi=-dd^{\C}\phi$, and $X_\phi=\nabla_\phi\phi$ is the gradient
of $\phi$ with respect to the K\"ahler metric
$\om_\phi(\cdot,J\cdot)$; see~\cite{CieEli12}. Since the space of
defining $J$-convex functions is contractible, different choices of
$\phi$ lead to homotopic Weinstein structures. Hence any Stein domain
$(W,J)$ has a canonically associated homotopy class $\fW(W,J)$ of
Weinstein structures on $W$. It is shown in \cite{CieEli12} that two
Stein domains $(W,J_0)$ and $(W,J_1)$ are Stein homotopic if and only
if $\fW(W,J_0)=\fW(W,J_1)$. 

A {\em Weinstein manifold structure} on an open manifold $V$ is a
triple $(\om,X,\phi)$ consisting of a symplectic form $\om$ on $V$, an
exhausting Morse function $\phi:W\to\R$, and a vector field $X$ on $V$
which is Liouville for $\om$, gradient-like for $\phi$, and complete
(i.e., its flow exists for all times).  
As in the Stein domain case, one can associate to a Stein manifold
$(V,J)$ and an exhausting $J$-convex function $\phi:V\to\R$ the
Weinstein structure
$\fW(V,J,\phi):=(\om_\phi=-dd^\C\phi,X_\phi=\nabla_\phi\phi,\phi)$
provided that $X_\phi$ is complete, which can always be achieved by
composing $\phi$ with a sufficiently convex function $\R\to\R$
(see~\cite[Section 11.5]{CieEli12}). By the {\em standard Weinstein
  structure} on $\C^n$ we mean the structure
$\fW_\st=(\om_\st,X_\st,\phi_\st):=\fW(\C^n,i,\frac{|z|^2}4)$. 
 
\begin{theorem}\label{thm:Weinstein} 
Let $(W,J)$ be a Stein domain, $\fW(W,J)$ the associated homotopy
class of Weinstein domain structures on $W$, and $f:W\into\C^n$ a
smooth embedding. Then $f$ is isotopic to a deformation equivalence
onto 
\begin{itemize}
\item[(a)] an $i$-convex domain if and only if the induced complex
  structure $f^*i$ is homotopic to $J$ through almost complex structures;
\item[(b)] a rationally convex domain if and only if in addition to
  (a) there is a representative $(\om,X,\phi)\in\fW(W,J)$ such that
  $f$ is isotopic to a symplectic embedding $\wt
  f:(W,\om)\into(\C^n,\om_\st)$;     
\item[(c)] a polynomially convex domain if and only if in addition to
  (a) and (b) the push-forward Weinstein structure $(\om_\st,\wt f_*X
  ,\phi\circ \wt f^{-1})$ extends to a Weinstein structure $(\wt\om
  ,\wt X,\wt\phi)\in\fW(\C^n,i)$ on the whole $\C^n$.   
\end{itemize}
\end{theorem}
 
 \begin{remark}
The proof of Theorem~\ref{thm:Weinstein} shows that the Weinstein
structure $(\wt\om,\wt X,\wt\phi)$ in (c) can be chosen with 
symplectic form $\wt\om=\om_\st$, standard at infinity, and homotopic
to the standard Weinstein structure via a homotopy fixed at infinity
and with fixed symplectic form.  
\end{remark}

A class of Stein domains  where conditions (b) and (c)  hold  are the {\em
flexible} Stein domains defined in~\cite{CieEli12}; we refer the
reader to there for their definition and properties. Let us just
mention that every subcritical Stein domain 
is flexible, and every domain which admits a Stein structure also
admits a flexible one which is unique (in a given homotopy class of
almost complex structures) up to Stein homotopy.  
For flexible Stein domains, Theorem~\ref{thm:Weinstein} implies
the following refinement of Theorems~\ref{thm:rational} 
and~\ref{thm:polynomial}.   

\begin{thm}\label{thm:flexible}
Let $(W,J)$ be a {\em flexible} Stein domain of complex dimension
$n\geq 3$, and $f:W\into\C^n$ a smooth embedding such that
$f^*i$ is homotopic to $J$ through almost complex structures. 
Then $(W,J)$ is deformation equivalent to a rationally convex
domain in $\C^n$. More precisely, $f$ is smoothly isotopic to an
embedding $g:W\into\C^n$ such that $g(W)\subset\C^n$ is rationally
convex, and $g^*i$ is Stein homotopic to $J$.  

If in addition $H_n(W;G)=0$ for every abelian group $G$,
then $g(W)$ can be made polynomially convex. 
\end{thm}

Theorems~\ref{thm:rational} and~\ref{thm:polynomial} follow directly
from Theorem~\ref{thm:flexible} and Theorems 13.1 and 13.5
in~\cite{CieEli12} which assert that, given a defining Morse
function $\phi:W\to\R$ on a domain $W$ of dimension $2n\geq 6$ without
critical points of index $>n$, any almost complex structure $J$ on $W$
is homotopic to an integrable complex structure $\wt J$ for which the
function $\phi$ (after composition with a convex increasing
function $\R\to\R$) is $\wt J$-convex and such that the Stein
structure $(\wt J,\phi)$ is flexible. 
\medskip

We conjecture that polynomial convexity in complex dimension $\geq 3$
implies flexibility. Note that, by~\cite[Theorem 15.11]{CieEli12}, this
conjecture would imply that if a polynomially convex domain
$W\subset\C^n$, $n\geq 3$, admits a defining Morse function without
critical points of index $\geq n$, then it is subcritical.  

For rational convexity, flexibility is not necessary. This follows,
for example, from another corollary of Theorem~\ref{thm:Weinstein}
which we now describe. Let $D^*L$ denote the unit cotangent disc bundle
of a closed $n$-dimensional manifold $L$ (with respect to some
Riemannian metric on $L$). By a theorem of Grauert~\cite{Gra58},
$D^*L$ carries a Stein structure $J_\Grauert$ which admits a defining
$J_\Grauert$-convex function $\phi_\Grauert$ whose Morse-Bott critical
point locus is the zero section $L$. The Stein domain
$(D^*L,J_\Grauert)$ is called a {\it Grauert tube} of $L$, and any two
Grauert tubes of $L$ are Stein homotopic through Grauert tubes. 
Note that the corresponding Weinstein structure
$\fW(D^*L,J_\Grauert,\phi_\Grauert)$ has the zero section $L$ as a
Lagrangian submanifold. 

\begin{cor}\label{cor:Lagrangian}
A Grauert tube $(D^*L,J_\Grauert)$ of a closed $n$-dimensional
manifold~$L$ 
\begin{itemize}
\item[(a)] is deformation equivalent to an $i$-convex domain in $\C^n$  
  if and only if $L$ admits a totally real
  embedding into $(\C^n,i)$;
\item[(b)] is deformation equivalent to a rationally convex
  domain in $\C^n$ if and only if $L$ admits a Lagrangian embedding
  into $(\C^n,\om_\st)$;
\item[(c)] is not deformation equivalent to a polynomially convex
  domain in $\C^n$.
\end{itemize}
\end{cor}

Note that part (c) is an immediate corollary of the vanishing
of $H_n(W;G)$ for polynomially convex domains, which was already stated
in Theorem~\ref{thm:polynomial}. 
The ``if'' in part (b) was proved by Duval and Sibony
in~\cite{DuvSib95}.  

Statement (b) shows that the question whether $D^*L$ is deformation
equivalent to a 
rationally convex domain in $\C^n$ depends on the topology on $L$ in a
subtle way. By a theorem of Gromov, the answer is negative for
manifolds with $H^1(L;\R)=0$. For example, since $S^3$ admits a
totally real embedding into $\C^3$ but no Lagrangian one,
$D^*S^3$ is deformation equivalent to an $i$-convex domain in $\C^3$,
but not to a rationally convex one.  

According to~\cite{BEE09}, any flexible Weinstein domain has vanishing
symplectic homology. On the other hand, by a result of several authors
(see~\cite[Section 17.1]{CieEli12} for references),
$(D^*L,J_\Grauert)$ has nonvanishing symplectic homology and is
therefore not flexible. Thus Grauert tubes of Lagrangian submanifolds
of $\C^n$ provide examples of rationally convex domains that are not
flexible. However, we do not know any example of a rationally convex
domain $W$ in $\C^n$ with $H^1(\p W;\R)=0$ which is not flexible. 

\subsection{Isotopy through $i$-convex domains}
One can ask when an $i$-convex domain $W\subset\C^n$ is
isotopic to a polynomially or rationally convex domain {\em via an
  isotopy through $i$-convex domains}. (Recall that in our terminology
an $i$-convex domain in $\C^n$ is a compact domain with smooth
strictly pseudoconvex boundary.) The answer is provided by

\begin{theorem} \label{thm:equiv-isotop}
Let $f_t:W\into \C^n$, $t\in[0,1]$, be an isotopy of smooth embeddings
of a domain $W$ such that $f_0(W)$ and $f_1(W)$ are $i$-convex. 
Then the path $f_t$ is 
homotopic with fixed end points in the space of embeddings to a path
of embeddings $\wt f_t:W\into\C^n$ onto $i$-convex domains $\wt
f_t(W)$ if and only if there exists a Stein homotopy $(W,J_t)$ such
that $J_0=f_0^*i$ and $J_1=f_1^*i$, and the paths $J_t$ and $f_t^*i$
are homotopic with fixed end points in the space of almost complex
structures on $W$. 
\end{theorem}

\begin{proof}
The ``only if'' follows simply by setting $J_t:=\wt f_t^*i$.

For the ``if'', pick a generic family of defining $J_t$-convex functions
$\phi_t:W\to\R$ and consider the Weinstein homotopy
$\fW_t:=\fW(W,J_t,\phi_t)$. By an ambient version of~\cite[Theorem
  15.2]{CieEli12}, after composing the $\phi_t$ with convex increasing
function $\R\to\R$, there exists a $2$-parametric family of embeddings
$h_{s,t}:W\into W$, $s,t\in[0,1]$, such that
\begin{itemize}
\item $h_{0,t}=h_{s,0}=h_{s,1}=\Id$;
\item the functions $\phi_t$ are $h^*_{1,t}f_t^*i$-convex;
\item the paths of Weinstein structures $\fW_t$ and
  $\fW(h_{1,t}^*f_t^*i,\phi_t)$ are homotopic with fixed end points
  and with fixed functions $\phi_t$.
\end{itemize}
The proof of this ambient version is essentially identical to the
proof of Theorem 15.2 given in~\cite{CieEli12}, replacing the
Parametric Stein Existence Theorem 13.6 by a $1$-parametric version of
the Ambient Stein Existence Theorem 13.4. 
Now the isotopy $\wt f_t:=f_t\circ h_{1,t}:W\into\C^n$ has the
required properties. 
\end{proof}
   
As a corollary of Theorem~\ref{thm:equiv-isotop},
Theorem~\ref{thm:flexible}, and~\cite[Theorem 
  15.14]{CieEli12} we get 

\begin{cor}\label{cor:flexible-isotopic}
(a) Any two {\em flexible} $i$-convex domains in $\C^n$, $n\geq 3$, that
are smoothly isotopic are isotopic through $i$-convex domains. 

(b) Every {\em flexible} $i$-convex domain $W\subset\C^n$, $n\geq 3$, is 
isotopic through $i$-convex domains to a rationally convex domain. 

(c) Every {\em flexible} $i$-convex domain $W\subset\C^n$, $n\geq 3$, 
satisfying $H_n(W;G)=0$ for every abelian group $G$ is 
isotopic through $i$-convex domains to a polynomially convex
domain. \qed 
\end{cor}

Without the flexibility hypothesis,
Corollary~\ref{cor:flexible-isotopic}(a) becomes false:

\begin{cor}
For every $n\geq 3$ there exist $i$-convex (even rationally convex)
domains in $\C^n$ that are smoothly isotopic, but not isotopic through
$i$-convex domains. 
\end{cor}

\begin{proof}
Let $L$ be a closed Lagrangian submanifold of $(\C^n,\om_\st)$. By
Corollary~\ref{cor:Lagrangian}(b), the Grauert tube
$(D^*L,J_\Grauert)$ is deformation equivalent to a rationally convex 
domain $W_0\subset\C^n$. On the other hand, by Theorems 13.1 and 13.5
in~\cite{CieEli12}, $D^*L$ carries a flexible Stein structure $J_{\rm
  flex}$, and by Theorem~\ref{thm:flexible}, $(D^*L,J_{\rm flex})$ is
deformation equivalent to a rationally convex domain
$W_1\subset\C^n$. Since $(D^*L,J_\Grauert)$ and $(D^*L,J_{\rm flex})$
are not deformation equivalent, $W_0$ and $W_1$ are not isotopic through
$i$-convex domains. 
\end{proof}

\subsection{Generalizations to other Stein
  manifolds}\label{sec:other-manifolds} 
The notions of rational and polynomial convexity generalize in a
straightforward way from $\C^n$ to a general Stein manifold
$(V,J)$. Let us denote by $\OO:=\OO(V,J)$ the algebra of holomorphic
functions on $(V,J)$, and by $\MM:=\MM(V,J)$ its field of fractions,
i.e., the algebra of meromorphic functions on $V$. For a compact set
$K\subset V$, one defines its 
{ \it $\OO$-hull} as 
$$
   \wh K_\OO  := \{z\in V\:\Bigl|\;
   |f(z)|\leq\mathop{\max}\limits_{u\in K}|f(u)|\hbox{ for  
  all functions }f\in\OO\},
$$
and its {\it $\MM$-hull} as
$$
   \wh K_\MM := \{z\in V\;\Bigl|\; |R(z)|\leq\mathop{\max}\limits_{u\in
     K}|R(u)|\hbox{ for all functions }R=\frac fg\in\MM,\; g|_K\neq
   0\}. 
$$
A compact set is called {\em $\OO$-convex} (resp.~{\em $\MM$-convex})
if $\wh K_\OO=K$ (resp.~$\wh K_\MM=K$).

Given a proper holomorphic embedding $(V,J)\into(\C^N,i)$, a
compact subset $K\subset V$ is $\OO$-(resp.~$\MM$-)convex if and only
its image in $\C^N$ is polynomially (resp.~rationally) convex. 
This follows from the standard corollary  of Cartan's Theorem  B (see e.g.~\cite[Corollary 5.37]{CieEli12}) 
that any holomorphic (resp. meromorphic) function on $V\subset\C^N$ is the restriction of a holomorphic 
(resp. meromorphic) function on $\C^N$, together with the fact that any holomorphic (resp. meromorphic) 
function on $\C^N$ can be approximated uniformly on compact sets by polynomials (resp. rational functions).
In particular, for $(V,J)=(\C^n,i)$ the notions of $\OO$- and $\MM$-convexity reduce
to polynomial and rational convexity, respectively.

Most of the results of this paper have analogues in this more general
situation. The proofs are   essentially identical, and we do not discuss them in this paper. In particular,
Theorems~\ref{thm:rational} and~\ref{thm:polynomial} generalize to 

\begin{theorem}\label{thm:convexity-in-Stein}
Let $(V,J)$ be a Stein manifold of complex dimension $n\geq 3$ and
$W\subset V$ be a compact domain. 

(a) $W$ is smoothly isotopic to an $\MM$-convex domain if and only if
it admits a defining Morse function without critical points of index
$>n$. 

(b) $W$ is smoothly isotopic to an $\OO$-convex domain if and only if
it satisfies, in addition, the following topological condition: 
\begin{itemize}
\item[($\mathrm{T_V}$)] The inclusion homomorphism $H_n(W;G)\to
  H_n(V;G)$ is injective for every abelian group $G$. 
\end{itemize}
\end{theorem}

{\bf Acknowledgement. }
We thank Stefan Nemirovski for motivating questions and extremely helpful
remarks and suggestions on the subject of this paper. We are also
grateful to Brian Conrad, Soren Galatius and John Morgan for their
help in finding the example discussed in Proposition~\ref{prop:top}(b) 
and for illuminating discussions of the involved algebraic issues.
Finally, we thank the anonymous referee for pointing out a simplified proof
of Lemma~\ref{lem:potential}. 

\section{Topological preliminaries}\label{sec:top-prelim}
 
In this section we deal with the topological parts of our results. We
begin with the proof of Proposition~\ref{prop:top}. The example in
part (b) is an adaptation of Example 4.35 from Hatcher's
book~\cite{Hat02}. 

\begin{proof}[Proof of Proposition~\ref{prop:top}]
(a) Clearly, the existence of a defining Morse function without
critical points of index $\geq n$ implies condition (T). Conversely,
suppose that $\phi:W\to\R$ is a defining Morse function without
critical points of index $>n$. If $W$ is simply connected and
$H_n(W)=0$, then Smale's theorem on the existence of Morse functions
with the minimal number of critical points \cite[Theorem 6.1]{Sma62}
allows us to cancel all index $n$ 
critical points against index $n-1$ critical points to obtain a
defining Morse function without critical points of index $\geq n$ 
 
(b) Fix $n\geq 3$. 
Let $W_0$ be the boundary connected sum of $S^{n-1}\times D^{n+1}$
and $S^1\times D^{2n-1}$, i.e., the domain obtained by connecting
$S^{n-1}\times D^{n+1}$ and $S^1\times D^{2n-1}$ by a $1$-handle. Note that $W_0$ 
is homotopy equivalent to $S^1\vee S^{n-1}$, so by~\cite{Hat02} it has
$\pi_{n-1}(W_0)\cong\Z[t,t^{-1}]$, the group ring of $\pi_1(W_0)=\Z$. 
Let $f:S^{n-1}\to W_0$ be a smooth map representing
$[f]=2t-1\in\pi_{n-1}(W_0)$. For dimension reasons, we can choose $f$
to be an embedding $f:S^{n-1}\into\p W_0$. This embedding has trivial
normal bundle. (To see this, note that the normal bundle of $f$
can be described by gluing two copies of $D^{n-1}\times\R^n$ via a map
$g:S^{n-2}\to O(n)$. Since $W$ can be embedded into $\R^{2n}$, the
normal bundle is stably trivial, so the homotopy class
$[g]\in\pi_{n-2}O(n)$ maps to zero under the stabilization map
$\pi_{n-2}O(n)\to\pi_{n-2}O$. As the stabilization map is an
isomorphism for $n\geq 3$, this shows that $[g]=0$ and thus the normal
bundle is trivial.) 

Let $W$ be the manifold obtained from
$W_0$ by attaching an $n$-handle along the embedding $f$, using any
framing. Since $W$ is built using handles of index $0,1,n-1,n$, it
carries a defining Morse function without critical points of index
$>n$. Moreover, the domain $W$ is homotopy equivalent to the space $X$ 
in~\cite[Example 4.35]{Hat02}, where it is shown that $H_i(X)=0$ for
all $i\geq2$. In particular, $H_n(W;G)=0$ for any coefficient group
$G$, so $W$ satisfies condition (T). Alternatively, this also follows
from the observation that attaching to $W$ a two-handle to kill its
fundamental group yields a ball. 

It remains to show that $\pi_n(W,\p W)\neq 0$. For this, we consider
the universal cover $\wt W$ of $W$. For subsets $A\subset W$, we
denote by $\wt A$ their preimage in $\wt W$. The group ring
$R=\Z[t,t^{-1}]$ of the fundamental group $\pi_1(W)=\Z$ acts (by deck
transformations) on the singular chain complex $C_*(\wt W)$, so the
homology groups $H_i(\wt W)$ are $R$-modules. The same applies to
relative homology groups $H_i(\wt A,\wt B)$ for subsets $B\subset
A\subset W$. Set $W_1:=W\setminus\Int W_0$ and consider the following
commuting diagram:
\begin{equation*}
\begin{CD}
   \pi_{n+1}(W,W_1) @>>> \pi_n(W_1,\p W) @>>> \pi_n(W,\p W) @>>>
   \pi_n(W,W_1) \\
   @V{\cong}VV @V{\cong}VV @V{\cong}VV @V{\cong}VV \\
   H_{n+1}(\wt W,\wt W_1) @>>> H_n(\wt W_1,\p\wt W) @>>>
   H_n(\wt W,\p \wt W) @>>> H_n(\wt W,\wt W_1) \\
   @V{\cong}VV @V{\cong}VV @V{\cong}VV @V{\cong}VV \\
   R @>{\phi}>> R @>>> R/\im(\phi) @>>> 0\;. \\
\end{CD}
\end{equation*}
Here the first two rows are parts of the long exact sequences of the
triple $(\wt W,\wt W_1,\p\wt W)$ (where we identify $\pi_i(\wt W,\wt
W_1)\cong\pi_i(W,W_1)$ etc), and the top vertical maps are Hurewicz
isomorphisms on the universal covers (which are simply connected). For
example, $W_1$ is obtained from $\p W$ by attaching an $n$-cell $e^n$,
so the pair $(\wt W_1,\p\wt W)$ is $(n-1)$-connected and the Hurewicz
map $\pi_n(W_1,\p W)\cong \pi_n(\wt W_1,\p\wt W)\to H_n(\wt W_1,\p\wt
W)$ is an isomorphism. Moreover, $H_n(\wt W_1,\p\wt W)\cong R$ is
generated as an $R$-module by $[e^n]$. Similarly, $W$ is obtained
from $W_1$ by attaching an $(n+1)$-cell $e^{n+1}$ plus higher
dimensional cells, so the pair $(\wt W,\wt W_1)$ is $n$-connected,
in particular $\pi_n(W,W_1)\cong H_n(\wt W,\wt W_1)=0$, and the
Hurewicz map $\pi_{n+1}(W,W_1)\cong \pi_{n+1}(\wt W,\wt W_1)\to
H_{n+1}(\wt W,\wt W_1)$ is an isomorphism. 
Moreover, $H_{n+1}(\wt
W,\wt W_1)\cong R$ is generated as an $R$-module by $[e^{n+1}]$.
By the five-lemma, the remaining Hurewicz map $\pi_n(W,\p W)\cong
\pi_n(\wt W,\p\wt W)\to H_n(\wt W,\p\wt W)$ is an isomorphism as well.  
Note that the homotopy groups in the diagram are also $R$-modules
and the Hurewicz maps are $R$-module homomorphisms. 

To compute the map $\phi$ recall that, by construction of the
attaching map $f$ above, the boundary map $\psi:R\cong H_n(\wt
W_1,\p\wt W_0)\to H_{n-1}(\wt W_0)\cong R$ is
multiplication by $(2t-1)$.   
By the duality lemma in~\cite[$\S$ 10]{Mil66}, the boundary map 
$\phi: R\cong H_{n+1}(\wt W,\wt W_1)\cong H_{n+1}(\wt W_0,\p\wt W_0)
\to H_n(\wt W_1,\p\wt W) \cong R$ is dual to $\psi$ in the sense that
$\phi$ is multiplication by $(-1)^{n}(2t^{-1}-1)$. Thus the image of
$\phi$ is the ideal $\la 2t^{-1}-1\ra$ generated by $2t^{-1}-1$, and
the above diagram shows $\pi_n(W,\p W)\cong \Z[t,t^{-1}]/\la
2t^{-1}-1\ra \neq 0$.  
\end{proof}

The following lemma will be used in the proof of
Theorem~\ref{thm:polynomial}. 

\begin{lemma}\label{lm:top}
Let $W\subset\C^n$, $n\geq 3$, be a domain satisfying condition (T) in
Theorem~\ref{thm:polynomial}, i.e., $W$ admits a defining Morse
function $\psi:W\to\R$ without critical points of index $>n$, and
$H_n(W;G)=0$ for every abelian group $G$. Then the function $\psi$
extends to a Morse function $\wt\psi:\C^n\to\R$ without critical
points of index $>n$ which equals $\wt\psi(z)=|z|^2$ outside a compact
set. 
\end{lemma}

\begin{proof}
Note that condition (T) implies $H_i(W;G)=0$ for all $i\geq n$ and any
abelian group $G$. Let us pick any gradient-like vector field $X$ for
$\psi$. 

Any map $g:S^{k-1}\to W$ from a sphere of dimension $k-1\leq
n-1$ is generically an embedding which does not meet any stable
manifold of the vector field $X$. Moreover, generically no trajectory
of $X$ is tangent to $g(S^{k-1})$ or intersects it in more than
one point. Hence, using the flow of $X$, $g$ is isotopic to an embedding
$f:S^{k-1}\into\p W$. We claim that $f$ is contractible in
$V:=\C^n\setminus\Int W$. For $k<n$ this is true since $\C^n$ is
obtained from $V$ by attaching handles of indices $\geq n$, and thus
$\pi_i(V)=\pi_i(\C^n)=0$ for all $i\leq n-2$. For $k=n$ it follows from
$$
    \pi_{n-1}(V)\cong H_{n-1}(V) \cong H_{n-1}(S^{2n}\setminus W) \cong H^n(W) \cong 0,
$$
where the first isomorphism follows from the Hurewicz theorem and $(n-2)$-connectivity of $V$,
the third one from Alexander duality (see~\cite[Theorem 3.44]{Hat02}), and the last one from
condition (T).  

So $f$ extends to an embedding $F:D^k\into \C^n\setminus\Int W$
transversely attaching the $k$-disk to $\p W$ along its boundary.  
Let $W'\subset\C^n$ be a tubular neighborhood of $W\cup
F(D^k)$ in $\C^n$ with smooth boundary, so $W'$ is obtained from $W$
by attaching a $k$-handle along $f$ (with some framing). Then:
\begin{enumerate}
\item $H_i(W';G)\cong H_i(W;G)$ for all $i>k$ and for every $G$;
\item $\pi_i(W')\cong \pi_i(W)$ for all $i<k-1$;
\item $\pi_{k-1}(W')$ equals $\pi_{k-1}(W)$ modulo the subgroup
  generated by $[f]$; 
\item $\pi_k(W')$ equals $\pi_k(W)$ if $[f]\in\pi_{k-1}(W)$ is non-torsion,
  and $\pi_k(W)\oplus\Z$ if $[f]$ is torsion; the same holds for $H_k$ in place of $\pi_k$. 
\end{enumerate} 
For property (iv), consider the long exact sequence
$$
   0=\pi_{k+1}(W',W)\to \pi_k(W) \into \pi_k(W')\stackrel{\phi}{\to} \pi_k(W',W) \stackrel{\psi}{\to}
   \pi_{k-1}(W). 
$$
Let $[F]$ be the generator of $\pi_k(W',W)\cong\Z$ that is mapped to $[f]$ under $\psi$. 
If $[f]$ is non-torsion, then $\im\,\phi = \ker\psi = \{0\}$ and thus $\pi_k(W)\cong\ker\phi = \pi_k(W')$. 
If $[f]$ is torsion, say $m[f]=0$ for some $m\in\N$, then $\im\,\phi=\ker\psi\cong\Z$ generated by
$m[F]$. In the resulting exact sequence $\pi_k(W)\into\pi_k(W')\stackrel{\phi}{\to}\Z=\la m[F]\ra\to 0$
the map $\phi$ has a right inverse (gluing the $k$-disk $mF$ to a $k$-disk in $W$ filling $mf$), 
so it follows that  $\pi_k(W')\cong\pi_k(W)\oplus\Z$. The proof for $H_k$ in place of $\pi_k$ is analogous.

Hence we can successively attach handles of indices $1,2,\dots,n-1$
to obtain a domain $W'\subset\C^n$ containing $W$ with
\begin{itemize}
\item $H_i(W';G)=0$ for all $i\geq n$ and for every $G$;
\item $\pi_i(W')=0$ for all $i<n-1$.
\end{itemize} 
By the Hurewicz theorem, $H_i(W';\Z)=0$ for all $i<n-1$ and
$H_{n-1}(W';\Z)\cong \pi_{n-1}(W')$. Moreover, by 
the universal coefficient theorem, $H_n(W';G)=0$ for every $G$ implies 
that $H_{n-1}(W';\Z)$ is torsion free. Hence we can attach
$n$-handles to $W'$ along spheres representing a basis of
$\pi_{n-1}(W')$ to obtain a domain $\wt W\subset\C^n$ containing $W$ with
\begin{itemize}
\item $H_i(\wt W;\Z)=0$ for all $i>n$;
\item $H_i(\wt W;\Z)\cong \pi_i(\wt W)=0$ for all $i\leq n-1$;
\item $H_n(\wt W;\Z)\cong H_n(W';Z)=0$. 
\end{itemize} 
Here the last condition follows from property (iv) above aplied to $H_n$. 
So $\wt W$ is simply connected and $H_i(\wt W)=0$ for all $i>0$. Since
$\dim\wt W\geq 6$, the $h$-cobordism theorem~\cite{Sma62} implies that
$\wt W$ is diffeomorphic to the closed ball $B^{2n}$. 

Now recall that $\wt W$ is obtained from $W$ by attaching handles of
indices $\leq n$. So we can extend the given defining Morse function
$\psi:W\to\R$ to a Morse function $\wt\psi:\wt W\to\R$ without
critical points of index $>n$ such that $\wt W=\{\wt\psi\leq
1\}$. Since $\wt W$ is diffeomorphic to $B^{2n}$, the embedding $\wt
W\into\C^n$ is isotopic to the standard embedding $B^{2n}\into\C^n$,
so we can extend $\wt\psi$ to an exhausting Morse function on $\C^n$
(still denoted by $\wt\psi$) without critical points outside $\wt W$,
and equal to $\wt\psi(z)=|z|^2$ at infinity. 
\end{proof}

\section{Complex analytic preliminaries}\label{sec:hol-prelim}

\subsection{Criteria for rational and polynomial convexity}\label{sec:criteria}
The proofs of Theorems~\ref{thm:rational} and~\ref{thm:polynomial} are
based on the following characterizations of rational 
and polynomial convexity. 
Consider the following condition on a 
$J$-convex domain $W$ in a complex manifold $(X,J)$: 
\begin{itemize}
\item[(R)] There exists a $J$-convex function $\phi:W\to\R$ such that
$W=\{\phi\leq 0\}$, and the form $-dd^\C\phi$ on $W$ extends to a
K\"ahler form $\om$ on the whole $X$.
\end{itemize}

The following criterion for rational convexity was proved by
S.~Nemirovski~\cite{Nem08} as a corollary of a result of J.~Duval and
N.~Sibony~\cite[Theorem 1.1]{DuvSib95}.  

\begin{criterion}\label{prop:rational}
An $i$-convex domain $W\subset\C^n$ is rationally convex if and only
if it satisfies condition {\rm (R)}. 
\end{criterion}

\begin{proof}
The ``if'' is the first proposition in~\cite{Nem08}. The following
proof of the ``only if'' was pointed out to us by S.~Nemirovski.
Let $W\subset\C^n$ be a rationally convex domain. Let
$\phi:U\to\R$ be an $i$-convex function on a bounded open neighborhood
$U$ of $W$ such that $W=\{\phi\leq 0\}$. Pick a cutoff function
$\rho:\C^n\to[0,1]$ which equals $0$ outside $U$ and $1$ on a smaller
open neighbourhood $V\subset U$ of $W$. 
Let $B\subset\C^n$ be a closed ball around the
origin containing $U$, and let $\psi:\C^n\to\R$ be such that
$-dd^\C\psi$ vanishes on $B$ and is strictly positive outside $B$. 
By~\cite[Theorem 2.1]{DuvSib95}, for every
$z\notin W$ there exists a nonnegative closed $(1,1)$-form $\om_z$
which vanishes on $W$ and is strictly positive on an open
neighbourhood $V_z$ of $z$. Finitely many such neighborhoods
$V_{z_1},\dots,V_{z_N}$ cover the compact set $B\setminus V$. Then for 
sufficiently large constants $c_i>0$,
$$
   \om := -dd^\C(\rho\phi) -dd^\C\psi + \sum_{i=1}^Nc_i\om_{z_i} 
$$
is a K\"ahler form with $\om|_W=-dd^\C\phi$. 
\end{proof}

We will also need the following classical criterion  for polynomial convexity which goes back to K.~Oka's paper~\cite{Oka53}
(see also ~\cite[Theorem 1.3.8]{Sto07}). 
  
\begin{criterion}\label{prop:polynomial}
An $i$-convex domain $W\subset\C^n$ is polynomially convex if
and only if there exists an exhausting $i$-convex function
$\phi:\C^n\to\R$ such that $W=\{\phi\leq 0\}$.
\end{criterion}
 
\begin{remark}\label{rem:T}
This criterion shows that to any polynomially convex domain
$W\subset\C^n$ we can attach handles of indices at most $n$ to obtain
a ball $B^{2n}\subset\C^n$. From the homology exact sequence of the
pair $(B^{2n},W)$,
$$
   0 = H_{n+1}(B^{2n},W;G)\to H_n(W;G) \to H_n(B^{2n};G)=0,
$$
we conclude that $H_n(W;G)=0$ for any coefficient group $G$. 
\end{remark}

We will also need the following lemma, where $\phi_\st(z):=|z|^2/4$
and $\om_\st:=-dd^\C\phi_\st$ denote the standard $i$-convex function and
K\"ahler form on $\C^n$. 

\begin{lemma}\label{lem:infty}
Let $A\subset\C^n$ be a compact subset and $B\subset\C^n$ a closed
ball with $A\subset\Int B$.  

(a) For every $i$-convex function $\phi:\C^n\to\R$ there exist an
$i$-convex function $\wt\phi$ which equals $\phi$ on $A$, and
$c\,\phi_\st$ outside $B$ for some constant $c>0$. 

(b) For every K\"ahler form $\om$ on $(\C^n,i)$ there exist a K\"ahler
form $\wt\om$ which equals $\om$ on $A$, and $c\,\om_\st$ outside $B$
for some constant $c>0$. 
\end{lemma}

\begin{proof}
(a) We pick a smooth function $\alpha:\R\to\R$ with the following
  properties:
\begin{itemize}
\item $\alpha'>0$, $\alpha''\geq 0$, and $\alpha(x)=cx$ for large $x$
  and some constant $c>0$;
\item $\alpha\circ\phi_\st<\phi$ on $A$, and
  $\alpha\circ\phi_\st>\phi$ near $\p B$.   
\end{itemize}
(To find $\alpha$, we choose a constant $b$ such that
$\phi_\st+b<\phi$ on $A$, then increase the function $x\mapsto x+b$
steeply near $\phi_\st^{-1}(\p B)$ to arrange
$\alpha\circ\phi_\st>\phi$ near $\p B$, and finally change it outside
$B$ to arrange $\alpha(x)=cx$ near infinity.) Then a smoothing of the
function $\max(\alpha\circ\phi_\st,\phi)$ gives an $i$-convex function
$\wt\phi$ with the desired properties.  

(b) By the $d$-Poincar\'e Lemma and the $\pbar$-Poincar\'e lemma on
$\C^n$ (see e.g.~\cite[Corollary 1.3.9]{Huy05}, and also the proof of
Lemma~\ref{lem:potential} below), we can write
$\om=-dd^\C\phi$ for some function $\phi:\C^n\to\R$. Then
$\wt\om:=-dd^\C\wt\phi$ for the function $\wt\phi$ provided by (a) has
the desired properties.  
\end{proof}

\begin{remark}\label{rem:constant-c} 
For any $c>0$ there exists a K\"ahler form $\om$ on $\C^n$ which
coincides with $c\om_\st$ on a ball $B$ of radius $R$ and with
$\om_\st$ outside a ball of the radius $2cR$.  Hence, in the statement
of Lemma \ref{lem:infty}(b) we can set $c=1$ if we are allowed to
choose the ball $B$ sufficiently large. 
\end{remark}

\subsection{Construction of K\"ahler potentials}\label{sec:potentials}

The following proposition provides K\"ahler
potentials with prescribed behavior on a Lagrangian submanifold. 

\begin{prop}\label{prop:potential}
Let $L$ be a real analytic Lagrangian submanifold (possibly noncompact
and/or with boundary) in a K\"ahler manifold $(X,J,\om)$ with real
analytic K\"ahler form $\om$, and let
$\rho:L\to\R$ be any real analytic function. Then there exists a
unique real analytic function $\phi$ on a neighborhood of $L$
satisfying  
$$
   -dd^\C\phi=\om,\qquad \phi|_{L}=\rho, \qquad
   (d^\C\phi)|_{L}=0. 
$$ 
\end{prop}

We denote by $D^n\subset\R^n\subset\C^n$ the closed unit disc, and by
$\Op D^n$ a sufficiently small (but not specified) open neighborhood
of $D^n$ in $\C^n$. Covering an arbitrary manifold $L$ by discs and
using uniqueness, Proposition~\ref{prop:potential} is an immediate
consequence of the following special case. 

\begin{lemma}\label{lem:potential}
Let $\om$ be real analytic real $(1,1)$-form on $\Op D^n$ with $d\om=0$ and
$\om|_{D^n}=0$, and let $\rho:D^n\to\R$ be any real analytic
function. Then there exists a unique real analytic function $\phi:\Op 
D^n\to\R$ satisfying 
$$
   -dd^\C\phi=\om,\qquad \phi|_{D^n}=\rho, \qquad
   \frac{\p\phi}{\p y_k}\Bigl|_{D^n}=0 \text{ for all }k=1,\dots,n. 
$$ 
\end{lemma}

The following simple proof was pointed out to us by the referee. 

\begin{proof}
For uniqueness, note that the difference of two solutions is a real
analytic function $\phi:\Op D^n\to\R$ satisfying 
$$
   -dd^\C\phi=0,\qquad \phi|_{D^n}=0, \qquad
   \frac{\p\phi}{\p y_k}\Bigl|_{D^n}=0 \text{ for all }k=1,\dots,n. 
$$ 
The last two conditions imply that near a point on $D^n$ we can write
$$
   \phi(x,y) = \phi_p(x,y) + O(|y|^{p+1}),\qquad
   \phi_p(x,y) = \sum_{|I|=p}a_I(x)y^I 
$$
for some $p\geq 2$. The first condition yields
$$
   0 = -dd^\C\phi = \sum_{k,\ell=1}^n\frac{\p^2\phi_p}{\p y_k\p
     y_\ell}dx_k\wedge dy_\ell + O(|y|^{p-1}),
$$
which implies $a_I(x)=0$ for all $I$. Thus $\phi$ vanishes to infinite
order along $D^n$, and unique continuation implies $\phi\equiv 0$. 

For existence, let $\om$ and $\rho$ as in the lemma be given.  By the $\p\pbar$-lemma (see e.g.~\cite{Huy05}) 
there exists a real analytic function $\psi:\Op D^n\to\R$ with $-dd^\C\psi=\om$. Let $f:\Op D^n\to\C$ 
be the holomorphic extension of $\rho-\psi$. Then the real analytic function $\psi':=\psi+\Re f$ satisfies
$\psi'|_{D^n}=\rho$ and $-dd^\C\psi'=-dd^\C\psi=\om$, where the last property follows from $\p\pbar=-\pbar\p$ 
and holomorphicity of $f$ via
$$
    -dd^\C\Re f = i\p\pbar(f+\bar f) = i(\p\pbar f-\pbar\p\bar f) = 0. 
$$
After renaming $\psi'$ back to $\psi$, we may thus assume $\psi|_{D^n}=\rho|_{D^n}$ and $-dd^\C\psi=\om$.
Since $d(d^\C\psi|_{D^n})=-\om|_{D^n}=0$, the Poincar\'e lemma (see e.g.~\cite{GuiPol74}) yields a
real analytic function $\vartheta:D^n\to\R$ with $d^\C\psi|_{D^n}=d\vartheta$. Let $g:\Op D^n\to\C$ 
be the holomorphic extension of $i\vartheta$. Then the real analytic function $\phi:=\psi+\Re g:\Op D^n\to\R$
still satisfies $-dd^\C\phi=\om$ (by the argument above) and $\phi|_{D^n}=\rho+\Re(i\vartheta)=\rho$. 
Moreover, holomorphicity of $g$ implies
\begin{align*}
   d^\C\phi|_{D^n} &= d^\C\psi|_{D^n} + \frac{1}{2}d^\C(g+\bar g)|_{D^n} 
   = d\vartheta + \frac{1}{2}(dg\circ i+d\bar g\circ i)|_{D^n} \cr   
   &= d\vartheta + \frac{1}{2}(i\circ dg - i\circ d\bar g)|_{D^n} 
   = d\vartheta + \frac{1}{2}\bigl(i\circ (id\vartheta) - i\circ (-id\vartheta)\bigr) \cr  
   &= d\vartheta - d\vartheta = 0. 
\end{align*}
This concludes the proof of Lemma~\ref{lem:potential}, and thus of Proposition~\ref{prop:potential}. 
\end{proof}

\subsection{Attaching isotropic discs to rationally convex
  domains}\label{sec:attaching-discs}

Using Proposition~\ref{prop:potential}, we now prove that rational
convexity persists under suitable attaching of isotropic discs.  
%
%
Given a $J$-convex domain $W$ in a complex manifold $(V,J)$, we say
that a totally real disc $\Delta\subset V\setminus\Int W$ is {\em
  $J$-orthogonally} attached to $\p W$ along $\p\Delta$ if
$J(T_x\Delta)\subset T_x(\p W)$ for all $x\in\p\Delta$.  
The following result is probably known to specialists,
though we could not find it in the literature; S.~Nemirovski has
informed us that he knew this fact. 

\begin{prop}\label{prop:rational-attaching}
Let $W$ be a $J$-convex domain in a complex manifold $(V,J)$. 
Suppose there exists a defining $J$-convex function
$\phi:W=\{\phi\leq c\}\to\R$ such that $-dd^\C\phi$ extends to a
K\"ahler form $\om$ on $V$. Let $\Delta\subset V\setminus\Int W$ 
be a real analytic $k$-disc, $J$-orthogonally attached to $\p W$ along
$\p\Delta$, such that $\om|_\Delta=0$.

Then for every open neighborhood $U$ of $W\cup\Delta$ there exists a
$J$-convex domain $\wt W\subset U$ with $W\subset\Int\wt W$ and a
defining $J$-convex function $\wt\phi:\wt W=\{\wt\phi\leq\wt c\}\to\R$
such that
\begin{itemize} 
\item $\wt\phi|_W=\phi$, and $\wt\phi$ has a unique index $k$ critical
  point in $\wt W\setminus W$ whose stable manifold is $\Delta$;
\item $-dd^\C\wt\phi$ extends to a K\"ahler form $\wt\om$ on $V$ which
  agrees with $\om$ outside $U$.  
\end{itemize}
\end{prop}

\begin{remark}
The real analyticity assumption on the disc $\Delta$ can probably be
removed using appropriate results on solutions of the $\dbar$-equation.
\end{remark}

The proof uses the following extension lemma. 

\begin{lemma}\label{lem:extension}
In the situation of Proposition~\ref{prop:rational-attaching}, there
exists an extension of $\phi$ to a $J$-convex 
function $\phi_1$ on a slightly larger domain $W_1=\{\phi_1\leq
c_1\}\subset U$ with the following properties:
\begin{enumerate}
\item $-dd^\C\phi_1$ extends to a K\"ahler form $\om_1$ on $V$ which
  agrees with $\om$ outside $U$ such that $\om_1|_\Delta=0$;
\item $d^\C\phi_1|_{\Delta\cap(W_1\setminus W)}=0$. 
\end{enumerate}
\end{lemma}

\begin{proof}
First, we extend the function $\phi$ to a $J$-convex function
$\wt\phi:\wt W\to\R$ on a slightly larger domain such that all level
sets of $\wt\phi$ in $\wt W\setminus W$ intersect $\Delta$
$J$-orthogonally, or in other words, $d^\C\phi|_{\Delta\cap(\wt
  W\setminus W)}=0$. 

To find a K\"ahler form extending $\wt\phi$, we argue as in the proof
of Lemma~\ref{lem:potential}. 
The closed $(1,1)$-form $\wt\om:=\om+dd^\C\wt\phi$ on $\wt W$ vanishes
on $W$. By the relative $d$-Poincar\'e lemma, we find a real $1$-form
$\lambda$ on $\wt W$ with $d\lambda=\wt\om$ and $\lambda|_W=0$. We
write $\lambda=\alpha+\bar\alpha$ for a $(1,0)$-form $\alpha$ on $\wt
W$, so that $\alpha|_W=0$ and 
$$
   \p\bar\alpha+\pbar\alpha=\wt\om,\qquad \p\alpha=0,\qquad
   \pbar\bar\alpha=0. 
$$
For sufficiently small $t\geq 0$ set $\Om_t:=\{\wt\phi<c+t\}\subset\wt
W$, so that $\Om_0=W$. Consider $\eps>0$ such that 
$\bar\Om_{3\eps}\subset\wt W$. By a result of H\"ormander and Wermer
(see~\cite[Theorem 8.36 and Remark 8.38]{CieEli12}), there exists a
smooth solution $\beta_\eps:\Om_{3\eps}\to\C$ of the equation
$\dbar\beta_\eps=\bar\alpha$ which satisfies for each integer $k\geq 0$ an 
estimate 
$$
   \|\beta_\eps\|_{C^k(\Om_{2\eps})}\leq 
   C_k\eps^{-n-k}\|\bar\alpha\|_{C^k(\Omega_{3\eps})}, 
$$
where $n=\dim_\C V$ and the constant $C_k$ depends on $k$ and the
diameter of the domain $\wt W$ but not on $\eps$. The function
$\psi_\eps:=\Im\beta_\eps:\Om_{2\eps}\to\R$ satisfies
$-dd^\C\psi_\eps=\wt\om|_{\Om_{2\eps}}$. Since $\bar\alpha|_W=0$, there
exists for every $k,N$ a constant $C_{k,N}$ depending on $k$ and $N$
but not on $\eps$ such that $\|\bar\alpha\|_{C^k(\Om_{3\eps})}\leq
C_{k,N}\eps^N$. It follows that
$$
   \|\psi_\eps\|_{C^k(\Om_{2\eps})}\leq \|\beta\|_{C^k(\Om_{2\eps})}\leq 
   C_{k,N}'\eps^{N-n-k} 
$$
with constants $C_{k,N}'$ not depending on $\eps$. Fix a cutoff
function $\rho:\R\to[0,1]$ with $\rho(t)=1$ for $t\leq 1$ and
$\rho(t)=0$ for $t\geq 2$ and define $\rho_\eps:V\to\R$ by
$\rho_\eps(x):=\rho\bigl(\frac{\phi(x)-c}{\eps}\bigr)$. Then the closed
$(1,1)$-form  
$$
   \om_1 := \om + dd^\C(\rho_\eps\psi_\eps)
$$
agrees with $\om+dd^\C\psi_\eps=-dd^\C\wt\phi$ on $\bar\Om_\eps$, and
with $\om$ outside $\Om_{2\eps}$. Moreover, for each $v\in TV$ with
$|v|^2=\om(v,Jv)=1$, the previous estimates yield
$$
   \om_1(v,Jv) \geq \om(v,Jv) -
   \|\rho_\eps\psi_\eps\|_{C^2(\Om_{2\eps})} \geq 1 - C_N'\eps^{N-n-2}
$$
with constants $C_N'$ not depending on $\eps$. Choosing $N:=n+3$ and
$\eps$ sufficiently small, we can arrange $\om_1(v,Jv)\geq 1/2$, so
$\om_1$ is a K\"ahler form and the restriction $\phi_1$ of $\wt\phi$
to the domain $W_1:=\bar\Om_\eps$ is the desired extension of $\phi$. 
\end{proof}

\begin{proof}[Proof of Proposition~\ref{prop:rational-attaching}]
{\bf Step 1. }Let $\phi_1:W_1=\{\phi_1\leq c_1\}\to\R$ and $\om_1$ be
the extensions provided by Lemma~\ref{lem:extension} and rename
$\phi_1,\om_1$ back to $\phi,\om$. We pick a value
$c'\in(c,c_1)$ and set $W':=\{\phi\leq c'\}$,
$\Delta':=\Delta\cap(V\setminus\Int W'\}$. After adding a constant to
$\phi$, we may assume without loss of generality that $c'=-1$. 

For the following, we need some notation from~\cite{CieEli12}. On
$\C^n$ with complex coordinates $z_j=x_j+iy_j$, we introduce the
functions 
$$
   r := \sqrt{x_1^2+\cdots+x_n^2+y_{k+1}^2+\cdots+y_n^2},\qquad
   R := \sqrt{y_1^2+\cdot+y_k^2},
$$
and for some fixed $a>1$ the standard $i$-convex function  
$$
   \psi_\st(r,R) := ar^2-R^2. 
$$
For $\eps>0$ we define the subsets
\begin{gather*}
   H_\eps := \{R\leq 1+\eps,\ r\leq\eps\}, \qquad
   H_\eps^y := H_\eps\cap\{x=0\}, \cr
   U_\eps := \{1-\eps\leq R\leq 1+\eps,\ r\leq\eps\}, \qquad 
   U_\eps^y := U_\eps\cap\{x=0\}, \cr
   D^k := \{R\leq 1,\ r=0\}.
\end{gather*}

{\bf Step 2. }Using~\cite[Theorem 5.53]{CieEli12}, we find a real
analytic embedding  
$f:H_\eps^y\into V$ such that $f(D^k)=\Delta'$, and $\phi\circ f$
agrees with $\psi_\st$ to first order along $\p D^k$. Moreover, by
property (ii) in Lemma~\ref{lem:extension} we can arrange that
$f^*d^\C\phi=0$ near $\p D^k$.  

We extend $f$ uniquely to a holomorphic embedding $F:(H_\eps,i)\into
(V,J)$ (for some possibly smaller $\eps>0$) and set $\psi:=\phi\circ
F:U_\eps\to\R$. Then for any $z\in\p D^k$ and $v\in T_zH_\eps^y=i\R^n$
we have $d_z\psi(v)=d_z\psi_\st(v)$ and
$d_z\psi(iv)=(f^*d^\C\phi)(v)=0=d_z\psi_\st(iv)$, so the $i$-convex
functions $\psi$ and $\psi_\st$ agree to first order at points of $\p
D^k$. According to~\cite[Proposition 3.26]{CieEli12}, there exists an
$i$-convex function $\vartheta:U_\eps\to\R$ which agrees with $\psi$
near $\p U_\eps$, and with $\psi_\st$ on $U_\delta$ for some
$\delta<\eps$. Moreover, according to ~\cite[Remark
3.27 (ii)]{CieEli12}, the conditions
$d^\C\psi|_{U_\eps^y}=d^\C\psi_\st|_{U_\eps^y}=0$ imply
$d^\C\vartheta|_{U_\eps^y}=0$.
Hence, after replacing $\phi$ by
$\vartheta\circ F^{-1}$ on $F(U_\eps)$ and shrinking $\eps$, we may
assume that $F^*\phi=\psi_\st$ on $U_\eps$. 

{\bf Step 3. }After real analytic approximations using~\cite[Theorem
  5.53]{CieEli12}, we may assume that $\phi$ and $\om$ are real
analytic on $F(H_\eps)$. Then $F^*\om$ is a real analytic K\"ahler form on $H_\eps$
which vanishes on $H_\eps^y$. By Lemma~\ref{lem:potential},
there exists a unique real analytic function $\psi':H_\eps\to\R$ (for
some possibly smaller $\eps$) satisfying 
$$
   -dd^\C\psi'=F^*\om,\qquad \psi'|_{H_\eps^y}=\psi_\st|_{H_\eps^y}, \qquad
   \frac{\p\psi'}{\p y_j}\Bigl|_{H_\eps^y}=0 \text{ for all }j=1,\dots,n. 
$$ 
Thus $\psi'$ agrees with $\psi_\st$ to first order along $H_\eps^y$. 
Moreover, since $-dd^\C\psi_\st=F^*\om$ on $U_\eps$ by Step 2,
uniqueness of $\psi'$ implies that $\psi'=\psi_\st$ on $U_\eps$. 

Again by~\cite[Proposition 3.26]{CieEli12}, there exists an
$i$-convex function $\vartheta':H_\eps\to\R$ which agrees with
$\psi'$ near $\p H_\eps$, and with $\psi_\st$ on $H_\gamma\cup U_\eps$
for some $\gamma<\eps$. Let $\phi':W'\cup F(H_\eps)\to\R$ be the $J$-convex
function which equals $\phi$ on $W'$, and $\vartheta'\circ F^{-1}$ on
$F(H_\eps)$. By construction, $-dd^\C\phi'$ extends by $\om$ to a
K\"ahler form $\om'$ on $V$ which equals $\om$ outside $U$. Moreover,
$F^*\phi'=\psi_\st$ on $H_\gamma$. 

{\bf Step 4. }According to~\cite[Corollary 4.4]{CieEli12}, there exists an
$i$-convex function $\wt\psi:H_\gamma\to\R$ with the following
properties:
\begin{itemize}
\item $\wt\psi=\psi_\st$ near $\p H_\gamma$;
\item $\wt\psi$ has a unique index $k$ critical point at the origin
  whose stable disc is $D^k$; 
\item $\wt\psi|_{D^k}<-1$. 
\end{itemize}
Let $\wt\phi:W'\cup F(H_\eps)\to\R$ be the $J$-convex function which equals 
$\wt\psi\circ F^{-1}$ on $F(H_\gamma)$ and $\phi'$ outside, and set
$\wt W:=\{\wt\phi\leq -1\}$. By construction, we have $\wt W\subset U$
and $W\subset\Int\wt W$. Moreover, $-dd^\C\wt\phi|_{\wt W}$
extends by $\om'$ to a K\"ahler form $\wt\om$ on $V$ which equals
$\om$ outside $U$. This concludes the proof of 
Proposition~\ref{prop:rational-attaching}.
\end{proof}

For the proof of Theorem~\ref{thm:Weinstein}, we will need the
following refinement of Proposition~\ref{prop:rational-attaching} in
the presence of an ambient Weinstein structure. 

\begin{prop}\label{prop:rational-attaching-Weinstein}
Let $(V,J)$ be a complex manifold equipped with a Weinstein structure
$(\om,X,\phi)$ such that $\om$ is a K\"ahler form for $J$. 
Let $W=\{\phi\leq c\}\subset V$ be a $J$-convex domain such that
$\phi|_W$ is defining and $J$-convex and
$\fW(W,J,\phi)=(\om,X,\phi)$ on $W$. 
Let $\Delta\subset V\setminus\Int W$ be a real analytic stable disc of
an index $k$ critical point $p$ of $(V,\om,X,\phi)$ on the first
critical level above $c$. 

Then for every open neighborhood $U$ of $W\cup\Delta$ there exists a
Weinstein structure $(\wt\om,\wt X,\wt\phi)$ on $V$ and a $J$-convex
domain $\wt W=\{\wt\phi\leq\wt c\}\subset U$ such that 
\begin{itemize}
\item $W\subset\Int\wt W$ and $(\wt\om,\wt X,\wt\phi)=(\om,X,\phi)$ on
  $W$; 
\item $\wt\phi|_{\wt W}$ is defining and $J$-convex and
$\fW(\wt W,J,\wt\phi)=(\wt\om,\wt X,\wt\phi)$ on $\wt W$;
\item $\wt\om$ is a K\"ahler form for $J$; 
\item the Weinstein structures $(\wt\om,\wt X,\wt\phi)$ and
  $(\om,X,\phi)$ agree outside $U$, have the same critical points, and
  are Weinstein homotopic via a homotopy with fixed critical points and
  fixed on $W$ as well as outside $U$.
\end{itemize} 
\end{prop}

\begin{proof}
As in Steps 1-3 of the proof of
Proposition~\ref{prop:rational-attaching}, we construct a $J$-convex
function $\phi':U'\to\R$ on a neighborhood $U'\subset U$ of
$W\cup\Delta$ with the following properties:
\begin{itemize}
\item $\phi'=\phi$ on $W\cup\Delta$;
\item $-dd^\C\phi'=\om|_{U'}$;
\item $p$ is the unique critical point (of index $k$) of $\phi'$ in
  $U'\setminus W$ with stable disc $\Delta$. 
\end{itemize}
According to~\cite[Proposition 12.14]{CieEli12}, there exists a 
Weinstein structure $(\om,X'',\phi'')$ on $V$ which agrees with
$(\om,X,\phi)$ outside $U$ and with $(\om,X',\phi')$ on a
neighborhood $U''\subset U'$ of $W\cup\Delta$ such that the Weinstein
structures $(\om,X'',\phi'')$ and $(\om,X,\phi)$ have the same
critical points and are Weinstein homotopic via a homotopy with fixed
critical points, fixed $\om$, and fixed on $W$ as well as outside $U$. 

Finally, we apply Step 4 of the proof of
Proposition~\ref{prop:rational-attaching} to modify $(\om,X'',\phi'')$
inside $U''$ to obtain the desired Weinstein structure $(\wt\om,\wt
X,\wt\phi)$ and $J$-convex domain $\wt W=\{\wt\phi\leq\wt c\}\subset
U$. 
\end{proof}

\section{Proofs of the main results}\label{sec:proofs}

\begin{proof}[Proof of Theorem~\ref{thm:Weinstein}]
(a) is a special case of~\cite[Theorem 13.4]{CieEli12}. 

(b) For the ``only if'', suppose that (after an isotopy)
$f(W)=\wt W\subset\C^n$ is rationally convex and $f^*i$ is Stein
homotopic to $J$. By Criterion~\ref{prop:rational}, there exists a
defining $i$-convex 
function $\wt\phi:\wt W\to\R$ such that $-dd^\C\wt\phi$ extends to a
K\"ahler form $\wt\om$ on $\C^n$. After applying Lemma~\ref{lem:infty}
and rescaling, we may assume that $\wt\om=\om_\st$ at infinity. Then
Moser's stability theorem yields a family of diffeomorphisms
$g_t:\C^n\to\C^n$ with $g_0=\id$, $g_t=\id$ at infinity, and
$g_1^*\om_\st=\wt\om$. Thus $g_t\circ f$ is an isotopy from $f$ to the
symplectic embedding $g_1\circ f:(W,f^*\wt\om)\into
(\C^n,\om_\st)$. Moreover, $f^*\wt\om$ is the symplectic form
of the Weinstein structure $\fW(W,f^*i,f^*\wt\phi)$, which belongs to
the class $\fW(W,J)$ by assumption. 

For the ``if'', suppose that (after an isotopy)
$f:(W,\om)\into(\C^n,\om_\st)$ is a symplectic embedding for some
$(\om,X,\phi)\in\fW(W,J)$. Applying
Proposition~\ref{prop:rational-attaching-Weinstein} inductively to
sublevel sets of the Weinstein structure
$(f_*\om=\om_\st,f_*X,f_*\phi)$ on $\bigl(f(W),i\bigr)$, we construct a 
Weinstein structure $(\wt\om,\wt X,\wt\phi)$ on $f(W)$ and an
$i$-convex domain $\wt W=\{\wt\phi\leq\wt c\}\subset \Int f(W)$ such
that  
\begin{enumerate}
\item $\wt\phi|_{\wt W}$ is defining and $i$-convex and
$\fW(\wt W,i,\wt\phi)=(\wt\om,\wt X,\wt\phi)$ on $\wt W$; 
\item $\wt\om$ is a K\"ahler form for $i$; 
\item the Weinstein structures $(\wt\om,\wt X,\wt\phi)$ and
  $(f_*\om=\om_\st,f_*X,f_*\phi)$ agree near $\p f(W)$, have the same
  critical points, and are Weinstein homotopic via a homotopy with
  fixed critical points and fixed near $\p f(W)$.  
\end{enumerate}
Thus $-dd^\C\wt\phi$ extends via $\wt\om$ to a K\"ahler form on
$f(W)$, and from there via $\om_\st$ to a K\"ahler form on $\C^n$, so
by Criterion~\ref{prop:rational} the domain $\wt W\subset\C^n$ is
rationally convex. Since $\wt\phi$ has no critical points in
$f(W)\setminus\Int\wt W=\{\wt c\leq\wt\phi\leq c\}$, we find a family
of diffeomorphisms $f_t:W\to W_t$ onto sublevel sets of $\wt\phi$ such
that $f_0=f$ and $W_1=\wt W$. Hence the pullback Weinstein structure
$\fW(W,f_1^*i,f_1^*\wt\phi)=f_1^*(\wt\om,\wt X,\wt\phi)$ is homotopic
to $f^*(\wt\om,\wt X,\wt\phi)$, which by property (iii) is homotopic to 
$(\om,X,\phi)\in\fW(W,J)$. By~\cite[Theorem 15.2]{CieEli12}, this
implies that the Stein structures $f_1^*i$ and $J$ are Stein
homotopic. 

(c) For the ``only if'', suppose that (after an isotopy)
$f(W)=\wt W\subset\C^n$ is polynomially convex and $f^*i$ is Stein
homotopic to $J$. By Criterion~\ref{prop:rational}, there exists an
exhausting $i$-convex function $\wt\phi:\C^n\to\R$ such that $\wt
W=\{\wt\phi\leq 0\}$. After applying Lemma~\ref{lem:infty},
we may assume that $\wt\phi=\phi_\st$ at infinity. Then
Moser's stability theorem yields a family of diffeomorphisms
$g_t:\C^n\to\C^n$ with $g_0=\id$, $g_t=\id$ at infinity, and
$g_1^*\om_\st=\wt\om:=-dd^\C\wt\phi$. Thus $g_t\circ f$ is an isotopy
from $f$ to the symplectic embedding $\wt f:=g_1\circ
f:(W,f^*\wt\om)\into (\C^n,\om_\st)$. Moreover, $f^*\wt\om$ is the
symplectic form of the Weinstein structure $\fW(W,f^*i,f^*\wt\phi)$,
which belongs to the class $\fW(W,J)$ by assumption. Finally, 
note that the push-forward Weinstein structure
$\wt f_*\fW(W,f^*i,f^*\wt\phi)$ extends to the Weinstein structure
$g_1^*\fW(\C^n,i,\wt\phi)$ on the whole $\C^n$ which is standard at
infinity. This Weinstein structure is homotopic via
$g_t^*\fW(\C^n,i,\wt\phi)$ to $\fW(\C^n,i,\wt\phi)$, and hence belongs
to the class $\fW(\C^n,i)$. 

For the ``if'', suppose that (after an isotopy)
$f:(W,\om)\into(\C^n,\om_\st)$ is a symplectic embedding for some
$(\om,X,\phi)\in\fW(W,J)$, and the push-forward Weinstein structure
$(f_*\om=\om_\st,f_*X,f_*\phi)$ extends to a Weinstein structure
$(\om_1,X_1,\phi_1)$ on the whole $\C^n$ which is 
homotopic to the standard Weinstein structure
$(\om_\st,X_0,\phi_0)=\fW(\C^n,i,\frac{|z|^2}{4})$ via a
Weinstein homotopy $(\om_t,X_t,\phi_t)$.
According to~\cite[Theorem 15.3]{CieEli12}, after composing the
$\phi_t$ with a convex increasing diffeomorphism $g:\R\to\R$, there
exists a family of diffeomorphisms $h_t:\C^n\to\C^n$, $t\in[0,1]$,
with $h_0=\id$ such that 
\begin{enumerate}
\item the functions $\phi_t\circ h_t^{-1}:\C^n\to\R$ are
  $i$-convex; 
\item the paths of Weinstein structures $(\om_\st,X_t,\phi_t)$ and
  $\fW(\C^n,h_t^*i,\phi_t)$ are homotopic with fixed functions
  $\phi_t$ and fixed at $t=0$.
\end{enumerate}
Thus $f_t:=h_t\circ f:W\into\C^n$ is a smooth isotopy from $f_0=f$ to
an embedding $f_1$ whose image $f_1(W)=h_1\bigl(f(W)\bigr)$ is a
sublevel set $\{\phi_1\circ h_1^{-1}\leq c\}$ of the $i$-convex
function $\phi_1\circ h_1^{-1}:\C^n\to\R$, and hence polynomially
convex by Criterion~\ref{prop:polynomial}. Due to property (ii) at  
$t=1$, the Weinstein structures $\fW\bigl(f(W),h_1^*i,\phi_1\bigr)$
and $(\om_1,X_1,\phi_1)$ on $f(W)=\{\phi_1\leq c\}$ are homotopic
with fixed function $\phi_1$. Therefore, the pullback Weinstein
structure $\fW(f_1^*i,\phi_1\circ
f)=f^*\fW\bigl(f(W),h_1^*i,\phi_1\bigr)$ is homotopic to
$f^*(\om_1,X_1,\phi_1)=(\om,X,\phi)\in\fW(W,J)$.  
By~\cite[Theorem 15.2]{CieEli12}, this
implies that the Stein structures $f_1^*i$ and $J$ are Stein
homotopic. 
\end{proof}
\medskip
 
\begin{proof}[Proof of Theorem~\ref{thm:flexible}]
The result is a consequence of Theorem~\ref{thm:Weinstein}
and the $h$-principle for Lagrangian caps in~\cite{EliMur13}. 
Let $(W,J)$ be a flexible Stein
domain of complex dimension $n\geq 3$, and $f:W\into\C^n$ a smooth
embedding such that $f^*i$ is homotopic to $J$ through almost complex
structures. Let $(\om,X,\phi)\in\fW(W,J)$ be an associated 
flexible Weinstein structure on $W$. 
According to~\cite[Corollary 6.3]{EliMur13}, the embedding $f$ is
isotopic to a symplectic embedding $\wt f:(W,\om)\into(\C^n,\om_\st)$. 
Hence, by Theorem~\ref{thm:Weinstein}(b), the embedding $f$ is
isotopic to a deformation equivalence onto a rationally convex domain. 
  
Now suppose in addition that $H_n(W;G)=0$ for every abelian group
$G$. Let $\wt\fW:=(\om_\st,\wt X:=\wt f_*X,\wt\phi:=\phi\circ\wt
f^{-1})$ be the push-forward Weinstein structure on $\wt W:=\wt
f(W)\subset\C^n$. 
According to Lemma~\ref{lm:top}, the defining function $\wt\phi:
\wt W\to\R$ extends to a Morse function $\wh\phi:\C^n\to\R$ without
critical points of index $>n$ which equals $\wh\phi(z)=|z|^2$ at infinity.
By~\cite[Theorem 13.1]{CieEli12}, we can extend the Weinstein
structure $\wt\fW$ to a flexible Weinstein structure
$\wh\fW=(\wh\om,\wh X,\wh\phi)$ on $\C^n$ such that the forms $\wh\om$
and $\om_\st$ are homotopic rel $\wt W$ as non-degenerate $2$-forms.  
According to~\cite[Theorem 14.5]{CieEli12}, the two flexible Weinstein
structures $\wh\fW$ and $\fW_\st$ are homotopic.
Hence, by Theorem~\ref{thm:Weinstein}(c), the embedding $f$ is
isotopic to a deformation equivalence onto a polynomially convex
domain.  
\end{proof}

For the proof of Corollary~\ref{cor:Lagrangian}, we need the following
lemma about Weinstein structures.

\begin{lemma}\label{lem:Lagrangian}
Let $\fW_t=(\om_t,X_t,\phi_t)$, $t\in[0,1]$, be a Weinstein homotopy on
$D^*L$ with Liouville forms $\lambda_t=i_{X_t}\om_t$ such that
$\fW_0=\fW(D^*L,J_\Grauert,\phi_\Grauert)$ is the Weinstein structure
associated to a Grauert tube.  
Then there exists an isotopy of
submanifolds $L_t\subset D^*L$ starting with the zero section $L_0$
such that $\lambda_t|_{L_t}$ is exact for all $t\in[0,1]$. 
\end{lemma}

\begin{proof} 
Set $W:=D^*L$ and consider the contact structures
$\xi_t:=\ker(\lambda_t|_{\p W})$ on its boundary. By Gray's stability
theorem, there exist diffeomorphisms $g_t:\p W\to \p W$ with
$g_t^*\xi_t=\xi_0$. After extending the $g_t$ to diffeomorphisms of
$W$ and replacing $\fW_t$ by their pullbacks, we may hence assume that
$\xi_t=\xi$ for all $t$. Pushing down by the flows of $-X_t$ for
suitable times, we then find embeddings $h_t:W\into W$ onto domains
with $W_t:=h_t(W)$ with $\p W_t$ transverse to $X_t$ such that
$h_t^*\lambda_t=h_0^*\lambda_0$ near $\p W$ for all $t$. Note that the
zero section $L_0=h_0(L_0)$ is contained in $W_0$ and
$h_0^*\lambda_0|_{L_0}=\lambda_0|_{L_0}=0$. 
 
By Moser's stability theorem (in the form stated in~\cite[Theorem
  6.8]{CieEli12}), there exist compactly supported diffeomorphisms
$f_t:W\to W$ with $f_0=\id$ such that
$f_t^*h_t^*\lambda_t-h_0^*\lambda_0=d\rho_t$ for compactly supported
functions $\rho_t:W\to\R$. Then the manifolds $L_t:=h_t\circ
f_t(L_0)\subset W_t\subset W$ satisfy
$$
   (h_t\circ f_t)(\lambda_t|_{L_t}) 
   = (f_t^*h_t^*\lambda_t)|_{L_0} 
   = (h_0^*\lambda_0+d\rho_t)|_{L_0} 
   = (d\rho_t)|_{L_0}, 
$$
and hence $\lambda_t|_{L_t}$ is exact for all $t\in[0,1]$. 
\end{proof}

\begin{proof}[Proof of Corollary \ref{cor:Lagrangian}]
Let $L$ be a closed $n$-dimensional manifold. 

(a) It is well known (see~\cite[Proposition 2.5]{CieEli12}) that every
totally real submanifold $L\cong L'\subset\C^n$ has a tubular
neighborhood which is a Grauert tube of $L$. 
Conversely, suppose that $(D^*L,J_\Grauert)$ is deformation equivalent
to an $i$-convex domain $W\subset\C^n$. 
Pick any defining $i$-convex function $\phi:W\to\R$. By
Lemma~\ref{lem:Lagrangian}, $W$ contains a submanifold diffeomorphic
to $L$ which is Lagrangian for $\om=-dd^\C\phi$, hence in particular
$i$-totally real. 

(b) is a special case of Theorem~\ref{thm:Weinstein}(b), in view of
Weinstein's Lagrangian neighborhood theorem and
Lemma~\ref{lem:Lagrangian}.  

(c) follows immediately from the vanishing of $H_n(W;G)$ for
polynomially convex domains stated in Theorem~\ref{thm:polynomial}. 
\end{proof}

\end{document}